\documentclass[12pt]{article}
  \usepackage[centertags]{amsmath}
  \usepackage{amsfonts}
  \usepackage{amssymb}
  \usepackage{amsthm}
  \usepackage{newlfont}

  \newlength{\defbaselineskip}
  \newcommand{\setlinespacing}[1]%
                               {\setlenght{\baselineskip}{#1 \defbaselineskip}}

  \newcommand{\supp}{{\textrm{supp}\, }}

  \theoremstyle{plain}

  \newtheorem{thm}{Theorem}[section]
  \newtheorem{cor}[thm]{Corollary}
  
  \newtheorem{pro}[thm]{Proposition}
  \newtheorem{rem}[thm]{Remark}
  
  \theoremstyle{definition}
  \newtheorem{defi}[thm]{Definition}
  \newtheorem{exm}[thm]{Example}

\numberwithin{equation}{section}

\begin{document}
\begin{center}
{\bf  Characterization Conditions and the Numerical Index}
\end{center}
\vspace{.15 cm}
\begin{center}
\small{Asuman G\"{u}ven AKSOY  and Grzegorz LEWICKI}
\end{center}

\mbox{~~~}\\
\mbox{~~~}\\
{\bf Abstract.} {\footnotesize In this paper we survey  some recent results concerning the numerical index $n(\cdot)$ for large classes of Banach spaces, including vector valued $\ell_p$-spaces and $\ell_p$-sums of Banach spaces where $1\leq p < \infty$.
In particular by defining  two conditions on a norm of a Banach space $X$, namely a Local Characterization Condition (LCC) and a Global Characterization Condition (GCC),
 we are able to show that if a norm on $X$ satisfies the (LCC), then $n(X) = \displaystyle\lim_m n(X_m).$  For the case in which  $ \mathbb{N}$
 is  replaced by a directed, infinite set $S$, we will  prove an analogous result  for
 $X$ satisfying the (GCC). Our approach  is motivated by the fact that $ n(L_p(\mu, X))= n(\ell_p(X)) = \displaystyle \lim_m n(\ell_p^m (X))$  \cite  {aga-ed-kham}.\\

\section{ Introduction}
Let $X$ be a Banach space over $\mathbb{R}$ or $\mathbb{C}$. We write $B_{X}$ for
the closed unit ball and $S_{X}$ for the unit sphere of $X$.
 The dual space is denoted by $X^{*}$ and the Banach algebra of all
continuous linear operators on $X$ is denoted by $B(X)$. For a linear subspace $Y$ of $X$ we denote by $ \mathcal{P}(X,Y)$ the set of all linear, continuous projections from $X$ onto $Y.$
\begin{defi}
The \textit{numerical range} of $T\in B(X)$ is defined by
$$W(T)= \{ x^{*}(Tx)  :~x\in S_{X},~x^{*}\in S_{X^{*}},~x^{*}(x)=1\}\cdot$$
The \textit{numerical radius} of $T$ is then given by
$$\nu (T)=\sup\{\vert \lambda\vert : ~\lambda\in W(T)\}\cdot$$
\end{defi}
Clearly, $\nu(\cdot)$ is a semi-norm on $B(X)$ and $ \nu(T) \le \Vert T\Vert$ for
all $T\in B(X)$.
 The \textit{numerical index} of $X$ is defined by
$$n(X)=\inf\{\nu(T) :~ T\in S_{B(X)}\}\cdot$$
Equivalently, the numerical index $n(X)$ is the greatest constant $k
\geq 0$ such that  $k\|T\| \leq \nu(T)$ for every $T \in B(X)$. The concept of numerical index was first introduced by Lumer
\cite{lg} in 1968. Since then, much attention has been paid to this
 equivalence constant between the numerical radius and the usual
norm in the Banach algebra of all bounded linear operators of a
Banach space.  It turns out that numerical index behave differently with respect to real or complex Banach spaces. In particular, it is known that  $0 \leq n(X) \leq 1$ if $X$ is a real space, and $\displaystyle\frac{1}{e}\leq n(X) \leq 1$ if $X$ is a complex space  \cite {JD-CM-JP-AW}. Furthermore, $n(X)> 0$ if and only if $\nu (\cdot)$ and $\|\cdot\|$ are equivalent norms.  For a Hilbert space $H$ of dimension greater than one, it is known that $n(H)= \displaystyle \frac{1}{2}$ in the complex case and $n(H) = 0$ in the real case. Classical references for this subject are the monographs by F. Bonsall and J. Duncan  \cite{bff-dj1} ,\cite{bff-dj2}  and the book of K. E. Gustafson and D. K. Rao \cite{gke-rdkm} for the Hilbert space case.   For more recent results relating to our discussion in this paper  we refer the reader to
\cite{aag-cbl}, \cite {aag-gl}, \cite{aga-ed-kham}, \cite{ee},
\cite{fc-mm-pr}, \cite{gke-rdkm}, \cite{vk-mm-jm-rp} \cite{lg-mm-pr}, \cite{mm}, \cite {mm-jm-mp-br} and \cite{pal}. The last detailed survey paper  on the  numerical index appeared in $2006$ by V. Kadets, at al,  \cite{VK-MM-RP}, which is a good source for open problems and references pertaining to the numerical index. In this paper we survey results on numerical index  emphasizing  progress since $2006$.
 However, our focus is motivated  by  results such as the one given  in \cite{aga-ed-kham} where the numerical index of vector-valued function spaces is considered  and a proof of $$ n(L_p(\mu, X)) = \displaystyle \lim_m n( (\ell_p^m (X))$$ is provided for a Banach space $X$ and for $1 \leq p < \infty$.  In \cite{AL} 
 above type of limit theorem for a class of Banach spaces including vector valued $\ell_p$ or $L_p$ spaces are obtained. 

 The study of the numerical index of absolute sums of Banach spaces is given in \cite{mm-jm-mp-br}, where under suitable conditions it is shown that the numerical index of a sum is greater or equal to the lim sup  of the numerical index of the summands (see Theorem 5.1 of \cite{mm-jm-mp-br}). Here we present a result which is an improvement over the limit theorem presented in \cite{mm-jm-mp-br}. 
 In \cite{AL}, we show the lim inf of the numerical index of the summands is greater or equal to the numerical index of the sum provided Banach space satisfies certain conditions. 
 
 We discuss what we mean by the  norm of a Banach space satisfying a condition called the Local Characterization Condition (LCC) or a condition called the Global Characterization Condition (GCC) and provide examples of Banach spaces satisfying the above mentioned conditions.
 We show if a norm on $X$ satisfies the local characterization condition, then
$$n(X) = \displaystyle\lim_m n(X_m)  \quad \mbox{and}\quad  n(X) = \displaystyle\lim_{s\in S} n(X_s)$$
 with the second equality holding true when  $X$ satisfies the (GCC) and where $S$ is any directed,  infinite set. 

\section{Characterization conditions and numerical index }

Given an arbitrary family of  $\{X_i \,\, i\in I\}$  of Banach spaces, we denote  $[ \oplus_{i\in I} {X_i}]_{c_{0}}$ (resp. $[ \oplus_{i\in I} {X_i}]_{\ell _{1}}]$ , \, $[ \oplus_{i\in I} {X_i}]_{\ell _{\infty}}$ ), the $c_0$-sum, (resp. $\ell_1$-sum, $\ell_{\infty}$-sum) of the family. In the case when $I$ consists of only two elements, we use the notation $X\oplus_{\infty} Y$ or $X\oplus_{1} Y$ . It is customary to use the notation $c_0(X), \ell_1(X)$ or $\ell_{\infty}(X)$ for countable copies of $X$.
First task of investigation is to check whether or not  numerical index of $c_0, \ell_1$, and $\ell_{\infty}$-sums can be computed in terms of the summands. The following proposition due to M. Martin and R. Paya (\cite{MP},Proposition 1), gives an affirmative answer.

\begin{pro}
Let   $\{X_i \,\, i\in I\}$  be a  set of Banach spaces. Then

$$n([\oplus_{i\in I} {X_i}]_{c_{0}}) = n([  \oplus_{i\in I} {X_i}]_{\ell _{1}})=n([  \oplus_{i\in I} {X_i}]_{\ell _{\infty}})  =  \displaystyle \inf _{i} n(X_i).$$
\end{pro}
The above proposition is not true for  $\ell_p$-sums if $p \neq 1, \infty$.  However, there is an interesting example given in (\cite{MP}, Example 2.b) which asserts the existence of a real Banach space $X$ for which the numerical radius is a norm but it is not equivalent to operator norm. In other words the numerical index of $X$ is zero even though the numerical radius $\nu (T) > 0$ for every $T\in B(X)$.


 For direct sums of Banach spaces, under some  general conditions it is shown that the numerical index of the sum is less or equal than the infimum of the numerical indices of the summands (see \cite{mm-jm-mp-br}, Theorem 2.1 and  Corollary 3.1). Furthermore, they also consider the numerical index of a Banach space which contains a dense increasing union of one-complemented subspaces and prove the following:
\begin{thm} (\cite{mm-jm-mp-br}, Theorem 5.1)
Let $X$ be a Banach space, let I be a directed set, and $\{X_i \,\, i\in I\}$  be an increasing family of one-complemented closed subspaces such that  $X=
\displaystyle\overline{ \bigcup_{ i \in I} X_i}$. Then,

$$ n(X) \geq\displaystyle \limsup_{i \in I} n(X_i).$$
 \end{thm}
 
 Later, in \cite {AL}, it is proved that the liminf of the numerical index of the summands is greater  than or equal to the numerical index of the sum if the Banach space satisfies a condition called the local characterization condition (LCC) or a condition called the global characterization condition (GCC).
 We show if a norm on $X$ satisfies the local characterization condition, then
$$
n(X) = \displaystyle\lim_m n(X_m)
$$
and
$$
n(X) = \displaystyle\lim_{s\in S} n(X_s)
$$
where $S$ is any directed, infinite set and $X$ satisfies the (GCC).
We also provide examples of spaces where (LCC) or (GCC) is satisfied.

The following theorem, which is a direct consequence of (\cite{pal}, Theorem 2.5), plays a crucial role in our further investigations.
\begin{thm}
\label{palacios}
Let $X$ be a Banach space over $\mathbb{R}$ or $\mathbb{C}$ and let
$$
\Pi(X) = \{ (x, x^*) \in S_X \times S_{X^*} : x^*(x) = 1 \}.
$$
Denote by $\pi_1$ the natural projection from $\Pi(X)$ onto $S_X$ defined by $ \pi_1(x,x^*) = x.$  Fix a set $ \Gamma \subset \Pi(X)$
such that $ \pi_1(\Gamma)$ is dense in $S_X.$ Then for any $ T \in B(X),$
$$
\nu(T) = \sup\{ |x^*(Tx)|: (x, x^*) \in \Gamma \}.
$$
\end{thm}
Applying the above theorem we can prove:
\begin{cor}
\label{crucial}
Let $ X$ be an infinite-dimensional Banach space and let $ Y \subseteq X$ be its linear subspace whose norm-closure is equal to $X.$
Define for $ L \in \mathcal{L}(X),$
\begin{equation}
\nu_Y(L) = \sup \{ |x^*Lx| : x^* \in S_{X^*}, x \in S_Y, x^*(y)=1\}.
\end{equation}
Then $ \nu(L) = \nu_Y(L).$
\end{cor}

\begin{defi}
\label{defLCC}
Let $X$ be a Banach space and $X_1\subset X_2\subset \cdots \subset X$ be its subspaces such that $X=
\displaystyle\overline{ \bigcup_{m=1}^{\infty}X_m}$.
Suppose for any $m\in \mathbb{N}$ there exists $P_m \in \mathcal{P}(X_{m+1}, X_m)$ with $\|P_m \| = 1$. We say the norm on $X$, $\parallel . \parallel_{X}$ satisfies the \emph{Local Characterization Condition} (LCC) with respect to $ \{ P_m\})_{m=1}^{\infty}$ if and only if for any $ m \in \mathbb{N}$ there exists  $D_m$ a dense subset of $S_{X_{m+1}}$ such that for any $x \in D_m$ there exists $x^* \in S_{X_{m+1}^*}$ a norming functional for $x$ in $X_{m+1}^*$  and a constant $b_m(x)\in \mathbb{R}_+$ such that $b_m(x) x^*\mid _{X_m}$ is a norming functional for $P_m x$ in $X_m^*.$ (In fact, if $ P_m(x) \neq 0, $ then $ b_m(x) = \|P_m(x)\|/x^*(P_m(x)).)$
\end{defi}


Now we present an example of a Banach space $X$ satisfying the (LCC) given in Definition  (\ref{defLCC}).
\begin{exm}
\label{main1}
Let for $ n \in \mathbb{N}$ $(Y_n, \| \cdot \|_n)$ be a Banach space. Set $X_1 = Y_1$ and $X_n = X_{n-1} \oplus Y_n.$
Let for $ n \in \mathbb{N},$ let $ p_n \in [1, \infty).$
Define a norm $| \cdot |_1$ on $X_1$ by $ |x|_1 = \|x\|_1$ and a norm $ |\cdot|_2$ on $X_2$ by
$$
|(x_1,x_2)|_2 = (\|x_1\|_1^{p_1}+ \|x_2\|_2^{p_1})^{1/p_1},
$$
where $x_i \in Y_i$ for $i=1,2.$ Then, having defined $ |\cdot |_n$ for $x=(x_1,...,x_n) \in X_n$, we can define $ | \cdot |_{n+1}$ on
$X_{n+1}$ by
$$
| (x,x_{n+1})|_{n+1} = (|x|_n^{p_n}+\|x_{n+1}\|_{n+1}^{p_n})^{1/p_n}.
$$
Note that if $x \in X_n,$ and $ m\geq n,$ then $ |x|_m = |x|_n.$
Let
$$
F= \{ \{y_n\}: y_n \in Y_n \hbox{ and } y_n = 0 \hbox{ whenever } n \geq m \mbox{ depending on } \{y_n\} \}.
$$
One can identify $F$ with $ \bigcup_{n=1}^{\infty} X_n,$ thus enabling us to define, for $x \in F,$ its norm as:
$$
\|x\|_F = \lim_n |x|_n,
$$
because for fixed $x \in F$ the sequence $ |x|_n$ is constant from some point on by the above mentioned property.
Notice that the completion of $ F$ (we will denote it by $X$) is equal to the space of all sequences $\{x_n\}$ such that  $x_n \in X_n$ and
$$
 \lim_n \|Q_nx\|_F = \sup_n \|Q_nx\|_F < +\infty ,
$$
where for $n \in \mathbb{N}$ and $ x=(x_1,x_2,...)$
$$
Q_n(x) = (x_1,...,x_n,0,...).
$$
Indeed, let $ \{x^s\}$ be a Cauchy sequence in $X.$ Notice that by definition of $ \| \cdot \|_F,$ $\|Q_n|_X\| =1.$ Hence for any
$ \epsilon >0,$ there exists $ N \in \mathbb{N}$ such that for any $s,k \geq N$ and $ n \in \mathbb{N},$
$$
|Q_n(x^s-x^k)\|_n \leq \epsilon.
$$
Consequently, for any $ n \in \mathbb{N},$ $Q_n(x^s)$ converges to some point in $ X_n.$ Hence for any $ i \in \mathbb{N}$
$ (x^s)_i \rightarrow x_i \in Y_i.$  Set $x=(x_1,x_2,...).$ Then, it is easy to see that $x \in X, $ since any Cauchy sequence is bounded and
$$
\|Q_n(x)\|_F = \lim_s\|Q_n(x^s)\|_F \leq \sup_s \|x^s\|_F < +\infty .
$$
Moreover, for fixed $\epsilon >0,$ for $s,k \geq N$ and any $ n \in \mathbb{N}, $
$$
\| Q_n(x^k - x^s)\|_F  \leq  \| x^s-x^k\|_F \leq  \epsilon.
$$
Hence fixing $ k \geq N$ and taking limit over s we get  for any $ n \in \mathbb{N},$
$$
\| Q_n(x^k -x)\|_F \leq \epsilon,
$$
and consequently $ \|x-x^k\|_X \leq \epsilon $ for $ k \geq N, $ which shows that $ \{ x^k\}$ converges to $x \in X.$
Hence $ X$ is a Banach space. Since for any $ x \in X,$ $ \lim_n\|Q_n(x)-x\|=0,$  $F $ is a dense subset of $X.$
Note that, for any $ n \in \mathbb{N}$, a map $P_{n}:X_{n+1} \rightarrow  X_n$ given by
$$
P_{n}(x_1,...,x_n,x_{n+1}) = (x_1,...,x_n,0),
$$
is a linear projection of norm one. By Definition (\ref{defLCC}) and the proof from Example (\ref{exCC}), the (LCC) is satisfied for the norm on $X.$
\end{exm}

\begin{defi}
\label{defGCC}
Let $X$ be a Banach space and let $\{ X_s\}_{s \in S}$ be a family of subspaces of $X$ such that $X=
\displaystyle\overline{ \bigcup_{s \in S} X_s}$. Assume that for any $s_1, s_2 \in S$ there exists $s_3 \in S$ such that $
X_{s_1} \cup X_{s_2} \subset X_{s_3}, $ i.e.  the family $\{ X_s\}_{s \in S}$ forms a directed set.
Suppose for any $s\in S$ there exists $P_s \in \mathcal{P}(X, X_s)$ with $\|P_s \| = 1$. We say the norm on $X$, $\parallel . \parallel_{X}$ satisfies the \emph{Global Characterization Condition} (GCC) with respect to $ \{ P_s\})_{s \in S}$ if and only if for any $ s \in S$ there exists  $D_s$ a dense subset of $S_{X}$ such that for any $x \in D_s$ there exists $x^* \in S_{X^*}$ a norming functional for $x$ in $X$  and a constant $b_m(x)\in \mathbb{R}_+$ such that $b_s(x) x^*\mid _{X_s}$ is a norming functional for $P_s x$ in $X_s^*.$ (In fact, if $ P_s(x) \neq 0, $ then $ b_s(x) = \|P_s(x)\|/x^*(P_s(x)).)$
\end{defi}
%

Now we present an example of a Banach space $X$ satisfying the condition (GCC) given in Definition(\ref{defGCC}).
\begin{exm}
\label{exCC}
Let $S$ be a directed and infinite set.
Fix $p \in [1,\infty).$ Let $X^p =(\displaystyle \oplus_{s \in S} X_s)_p$ be the direct, generalized $l^p$-sum of Banach spaces $(X_s, \parallel . \parallel_s)_{s \in S}$, defined as
$$
X^p  = \{ (x_s)_{s \in S}: x_s \in X_s, card(\supp((x_s)_{s \in S})) \leq \aleph_o \mbox{ and } \displaystyle\sum_{s \in S} ( \parallel x_s \parallel_s)^p < \infty \},
$$
where $\supp((x_s)_{s \in S})= \{ s \in X_s : x_s \neq 0\}.$
Clearly, the norm of  $x \in X^p$ is
$$
\parallel x \parallel =  (\displaystyle\sum_{s \in S} ( \parallel x_s \parallel_s)^p)^{1/p}
$$
and in case $ S = \mathbb{N}$ and  $X_i=X$
 for all $i \in \mathbb{N}$, $X^p= \ell_p(X)$.
Fix any finite set $W  \subset S.$
Next, consider spaces $Z_W =   \oplus_{s \in W} X_s$ and the projections
$$
P_W((x_s)_{s \in S}) = (z_s)_{s \in S},
$$
where $z_s = x_s$ for $s \in W$ and $ z_s =0$ otherwise. Let $ F = \{ W \subset S: card((W) < \infty \}.$
Now we show that the (GCC) is satisfied for $X$ ,$\{ Z_W \}_{W \in F}$ and $\{ P_W \}_{W \in F}$. It is obvious that $ \| P_W\| =1$ for any finite subset $W$ of $S$ and
$p\in [1, \infty).$ Now assume that $ 1 < p < \infty.$ To show that the characterization condition is satisfied for the norm on $X$, note that for any $x \in X^p\setminus\{0\} $ there exists a norming functional of the form
$$
x^* = \displaystyle \frac{\left( \parallel x_s \parallel_s^{p-1} x_s^* ( . )\right )_{s \in S}}{\displaystyle(\sum_{s \in S} ( \parallel x_s \parallel_s)^p)^{\frac{p-1}{p}}}
$$
where $x_s^* \in X_s^*$ is a norming functional for $x_s \in X_s$. Setting $C = \displaystyle(\sum_{s \in S} ( \parallel x_s \parallel_s)^p) ^{\frac{p-1}{p}}$, to see that  $\parallel x^* \parallel \leq 1$, let $y\in X$ be an element with $\parallel y \parallel =1$, then
$$
|x^* (y)| = | \displaystyle \frac{\sum_{s \in S} \left( \parallel x_s \parallel_s\right)^{p-1} x_s^*(y_s)}{C}| \leq  \frac{1}{C}\displaystyle\sum_{s \in S}
 \left( \parallel x_s \parallel_s)^{p-1}\right) |x_s^*(y_s)|.
$$
Applying the H\"{o}lder inequality with conjugate pairs $p$ and $q$:
$$ |x^* (y)| \leq \frac{1}{C} \left[\displaystyle (\sum _{s \in S} \parallel x_s\parallel_{s}^{p-1})^q\right]^{\frac{1}{q}} . \left[\displaystyle\sum_{s \in S}\parallel y_s \parallel_s^p\right]^{\frac{1}{p}}.
$$
Since $q =\displaystyle \frac{p}{p-1}$ and $\parallel y \parallel =\left[\displaystyle\sum_{s \in S}\parallel y_s \parallel_s^p\right]^{\frac{1}{p}} = 1$ we have $|x^*(y)| \leq 1$.
It is easy to see that $x^*$ is a norming functional for $x$ because
$$
x^* (x)= \frac{1}{C}\displaystyle \sum_{s \in S} \parallel x_s \parallel_{s}^{p-1} x_s^*(x)= \displaystyle \frac{\parallel x \parallel^p}{\parallel x \parallel^{p-1}}= 1.
$$
Furthermore, if $ P_W \neq 0,$ from
$$ (P_Wx)^* = \displaystyle\frac{\left( \parallel x_w \parallel_w^{p-1} x_w^*(.)\right)_{w \in W}}{ \left(\displaystyle \sum_{w\in W} \parallel x_w \parallel_w^p \right)^{\frac{p-1}{p}}}
$$
and  writing $x^*\mid_{Z_W}$,  we obtain that $b_W(x)= \displaystyle \frac{\parallel x \parallel^{p-1}}{\parallel P_W x\parallel^{p-1}}$.
If $p=1$ then for any $x \in X^1\setminus\{0\} $ there exists a norming functional of the form $(x_s^* ( . ))_{s \in S}$,  where $x_s^* \in X_s^*$ is a norming functional for $x_s \in X_s.$ It is easy to see that
$$
 \|(x_s^* ( . ))_{s \in S}\| = \sup_{s \in S} \|x_s^*\|.
$$
Reasoning as in the previous case we get that (GCC) is satisfied for $p=1.$
\end{exm}

\begin{rem}
Note that if $ S = \mathbb{N}$ and $ X_s \subset X_z$ for $ s, z \in \mathbb{N},$ $ s \leq z$ then the  (GCC) implies the  (LCC).
\end{rem}
The above definition is motivated by the space $X= \ell_p$ with $1< p < \infty,$ $ X_m =\ell_p^{(m)}$ and a sequence of projections
$\{ P_m\}_{m=1}^{\infty}$ defined by
$$
P_m (x_1,...,x_m,x_{m+1},...) = (x_1,...,x_m,0,...).
$$
For $x\neq 0$ and $x \in \ell_p$,
the form of the norming functional is $x^* = \displaystyle \frac{(|x_i|^{p-1} sgn(x_i))}{\|x\|_{p}^{p-1}}$ and clearly
$$
x^*\mid_{X_m} = \displaystyle \frac{(|x_i|^{p-1} sgn(x_i))}{\|x\|_{p}^{p-1}} \,\,\,\mbox{where}\,\,\,i\in\{1,2,\ldots ,m\}
$$
and the norming functional for $P_mx$, $(P_mx)^*$ takes the form
$$
(P_m x)^* = \displaystyle \frac{(|x_i|^{p-1} sgn(x_i))}{\|P_m x\|_{p}^{p-1}}\,\,\,\mbox{where}\,\,\, b_m(x) =\displaystyle \frac{\|x\|_{p}^{p-1}}{\|P_m x\|_{p}^{p-1}} .
$$
The above (GCC) is also satisfied for norms of $\ell_1$ and $c_0$ (with the same sequence $\{ P_m\}_{m=1}^{\infty}$).

We start by investigating first some consequences of (LCC).
\begin{pro} \cite{AL} 
\label{one}
Let $X$ be a Banach space satisfying (LCC) with respect to $ \{ P_m\})_{m=1}^{\infty}.$
For a fixed $m\in \mathbb{N}$ and $L \in \mathcal{L}(X_m)$,define a sequence
$$
w_m(L)= \nu(L),\,\, w_{m+1}(L) = \nu(L\circ P_m), \ldots, w_{m+j} (L) = \nu (L \circ Q_{m,j}),
$$
where $Q_{m,j}= P_m \circ \ldots \circ P_{m+j-1}.$ (For $ j \geq 1$ $\nu (L \circ Q_{m,j})$ denote the numerical radius of $ L \circ Q_{m,j}$ with respect to  $X_{m+j}.)$
Then
$ \nu(L)=w_m(L)=w_{m+j} (L)$ for $j=1,2,\ldots$.
\end{pro}
\begin{proof}
Since $ X_m \subset X_{m+1}$ for any $m \in \mathbb{N},$ it is easy to see that $w_{m+j} (L)$ is an increasing sequence with respect to $j$, since
$$
w_{m+j}(L)= \sup\{ |x^* L \circ Q_{m,j} x |:\,\,x\in S_{X_{m+j}},\,\, x^*\in S_{X_{m+j}^*}, x^*(x)=1\}
$$
$$
\leq \sup\{ |x^* L Q_{m,j}P_{m+j} x |:\,\,x\in S_{X_{m+j+1}},\,\, x^*\in S_{X_{m+j+1}^*}, x^*(x)=1\} =w_{m+j+1}(L).
$$
Now we prove that $w_m = w_{m+1}.$ To do this for any $x \in D_m$ select $x^*_x  \in S_{X_{m+1}^*}$ satisfying the requirements of Definition(\ref{defLCC}).
Set
$$
 \Gamma_m = \{(x,x^*_x)\in \Pi(X): x \in D_m\} .
$$
Note that by Def.(\ref{defLCC}),
$$
\frac{b_m(x)}{\|P_m(x)\|} = \frac{1}{x^*(P_m(x))} \geq 1.
$$
Hence  for any $(x,x^*_x) \in \Gamma_m,$
$$
 | x^*_x \circ L \circ P_m x |  =| (x^*_x)|_{X_m}\circ  L \circ P_m x |  \leq \frac{b_m(x)}{\parallel P_mx \parallel }|(x^*_x)|_{X_m} \circ L
 \circ P_m x |
$$
$$
= |(b_m(x) x^*_x)|_{X_m} L (\frac{P_m x }{\parallel P_m x \parallel}) | \leq \nu(L).
$$
Notice that by Def.(\ref{defLCC}), $ \pi_1(\Gamma_m) = D_m$ and $D_m$ is dense in $ S_{X_{m+1}}.$ By Theorem (\ref{palacios})
applied to $ \Gamma_m$ and $ L \circ P_m,$
$$
w_{m+1}(L) = \nu(L\circ P_m) \leq \nu(L) = w_m(L)
$$
and thus $w_m(L)= w_{m+1}(L)$. Induction on $j$ results in $w_m(L)= w_{m+j}(L)$.
\end{proof}Next, we examine some properties of the numerical index when the norm of the space $X$ satisfies the LCC. Proof of the following proposition  is given in \cite{AL}.
\begin{pro}  The following hold true:
\renewcommand{\labelenumi}{\alph{enumi})}
\begin{enumerate}
\item 
\label{projections}
Let $P_j \in \mathcal{P}(X_{j+1}, X_j) $ with $\parallel P_j \parallel = 1$. For a fixed $m \in \mathbb{N}$, define projections $Q_{m,j} \in \mathcal{P}(X_{m+j}, X_m) $ as
$ Q_{m,j} = P_m \circ P_{m+1} \circ \cdots \circ P_{m+j-1}$. Then
$$
 \displaystyle \lim_{j \rightarrow\infty} Q_{m,j} = Q_m
$$
where $Q_m \in \mathcal{P}(X, X_m) $ with $\parallel Q_m \parallel = 1$ and $X= \displaystyle\overline{ \bigcup_{m=1}^{\infty}X_m}$.
\item
\label{inequality}
For a fixed $m\in \mathbb{N}$ and $L \in \mathcal{L}(X_m)$ we have
$$
w_{m+j} (L) \leq \nu(L \circ Q_m)
$$
for all $j,$ where $\nu(L \circ Q_m)$ denotes the numerical radius of $L \circ Q_m$ with respect to $X.$
\item
\label{interesting} Let $X$ satisfy (LCC) with respect to $ \{ P_m \}_{m=1}^{\infty}.$ Then
for any $ m \in \mathbb{N}$  and $L \in \mathcal{L}(X_m), $
$$
\nu(L) = \nu(L \circ Q_m),
$$
where $Q_m \in \mathcal{P}(X,X_m)$ are defined in Proposition(\ref{projections}).

\item 

\label{LCC}
Assume that $ \| \cdot \|_X$ satisfies (LCC). Then for any $ m \in \mathbb{N}, $
$$
n(X_m) \geq n(X).
$$

\end{enumerate}
\end{pro}

\begin{thm} \cite{AL}
\label{main}
Let $X$ and $X_m $ and $ P_m$ be as in Definition (\ref{defLCC}). Then
$$
n(X) = \displaystyle \lim_m n(X_m).
$$
\end{thm}
\begin{proof}
By  the above Proposition, part d) we have,  $ n(X_m) \geq n(X)$ for any $ m \in \mathbb{N}.$ Hence,
$$
\liminf_m n(X_m) \geq n(X).
$$
By Theorem 5.1 of \cite{mm-jm-mp-br}, we already know that
$$
n(X) \geq \limsup_m n(X_m),
$$
which proves the equality.
\end{proof}
Now we present our main theorem.
\begin{thm} \cite{AL}
\label{main2}
Let $X$ and $X_s $ and $ P_s$ be as in Definition (\ref{defGCC}). Then
$$
n(X) = \displaystyle \lim_s n(X_s).
$$
\end{thm}
\begin{proof}
By Theorem 5.1 of \cite{mm-jm-mp-br}, we already know that
$$
n(X) \geq \limsup_s n(X_s).
$$
Now we prove that
$$
\liminf_s n(X_s) \geq n(X).
$$
Fix $ s \in S$ and $ L \in \mathcal{L}(X_s).$
We show that $ \nu(L) \geq \nu(L \circ P_s),$ where  $ \nu(L) $ denotes the numerical radius of $L$ with respect to $X_s$ and  $ \nu(L \circ P_s) $ denotes the numerical radius of $L\circ P_s$ with respect to $X.$ To do that, for any $x \in D_s$ select $x^*_x  \in S_{X^*}$ satisfying the requirements of Definition (\ref{defGCC}). Let
$$
 \Gamma_s = \{(x,x^*_x)\in \Pi(X): x \in D_s\} .
$$
Observe that for any $ s \in S,$
$$
\frac{b_s(x)}{\|P_s(x)\|} = \frac{1}{x^*(P_s(x))} \geq 1.
$$
Note that by Def.(\ref{defGCC}), for any $(x,x^*_x)\in \Gamma_s,$
$$
 | x^*_x \circ L \circ P_s x |  =| (x^*_x)|_{X_s}\circ  L \circ P_s x |  \leq \frac{b_s(x)}{\parallel P_sx \parallel }|(x^*_x)|_{X_s} \circ L
 \circ P_s x |
$$
$$
= |(b_s(x) x^*_x)|_{X_s} L (\frac{P_s x }{\parallel P_s x \parallel}) | \leq \nu(L).
$$
Notice that by Def.(\ref{defGCC}), $ \pi_1(\Gamma_s) = D_s$ and $D_s$ is dense in $ S_{X}.$ By Theorem (\ref{palacios})
applied to $ \Gamma_s$ and $ L \circ P_s,$
$\nu(L\circ P_s) \leq \nu(L),$ as required.
Hence we get immediately that
$$
n(X_s) = \inf \{ \nu(L) : L \in \mathcal{L}(X_s), \| L \| =1 \} = \inf \{ \nu(W) : L \in \mathcal{L}(X), \| W \| =1 \} = n(X).
$$
Consequently $\liminf_s n(X_s) \geq n(X)$ and finally  $\lim_s n(X_s) = n(X),$ as required.
\end{proof}

\section{Computing the numerical index}
Let $L_p(\mu, X)=L_p(\mu)$ denote the classical Banach space of p-integrable functions $f$ from $\Omega$ into $X$ where $(\Omega, \Sigma, \mu)$ is a given measure space. As usual $\ell_p(X)$ denotes the Banach space of all $x=(x_n)_{n \geq 1}, \,\,\, x_n \in X$, such that $\displaystyle \sum_{n=1}^{\infty} ||x_n||^p < \infty$. Finally $\ell_p^m(X)$ is the Banach space of finite sequences $x=(x_n)_{ 1 \leq n\leq m}, \,\, x_m\in X$ equipped with the norm $(\displaystyle \sum_{n=1}^{m} ||x_n||^p)^{\frac{1}{p}}  $.

The numerical index of some vector valued function spaces  are known, as shown in the following theorem.

\begin{thm} (\cite{MP}, \cite{MM-ARV})
Let $K$ be a compact Hausdorff space and $\mu$ be a positive $\sigma$-finite measure. Then

$$ n(C(K,X))=n(L_1(\mu, X))= n(L_{\infty}(\mu, X))=n(X)$$
where by $C(K,X)$ we mean the space of $X$-valued continuous functions on $K$, and by $L_1(\mu, X), \,\,\mbox{and}\,\,(L_{\infty}(\mu, X)$ we denote respectively the space of $X$-valued $\mu$-Bochner-integrable functions and the space of  $X$-valued $\mu$-Bochner-measurable  and essentially bounded functions.

\end{thm}

One of the most important families of classical Banach spaces whose numerical indices remain unknown is the family of $L_p$-spaces when $p \neq 1,2, \infty$. This stayed as an open problem till $2005$  when Ed-Dari  \cite {ee} and Ed-Dari and Khamsi \cite {ee-2}  made some progress proving:
\begin{thm} (\cite{ee}, \cite{ee-2})

Let $1\leq p \leq \infty$ be fixed. Then 
\renewcommand{\labelenumi}{\alph{enumi})}
\begin{enumerate}
\item $n(L_p[0,1])=n(\ell_p) = \inf\{n(\ell_p^m) : \,\, m \in \mathbb{N} \}$, and the sequence $\{n(\ell_p^m) \}_{m \in \mathbb {N}}$ is decreasing.
\item $n(L_p(\mu) )\geq  n(\ell_p) $ for every positive measure $\mu$.
\item In the real case
$$\displaystyle \frac{1}{2} M_p \leq n(\ell_p^m) \leq M_p, \,\,\, \mbox{where}\,\,\, M_p=\sup_{t\in [0,1]} \displaystyle \frac{|t^{p-1} -t|}{1+t^p}.$$
\end{enumerate}
\end{thm}
In $2007$ Ed-Dari's result in \cite{ee} was extended for the vector valued functions \cite {aga-ed-kham}. 

\begin{thm} \cite{aga-ed-kham}
Let $X$ be a Banach space. Then for every real number $p$, $1 \leq p \leq \infty$, the numerical index of the Banach space $\ell_p(X)$ is given by 
$$ n(\ell_p(X)) = \displaystyle \lim_m n(\ell_p^m (X)).$$
\end{thm}

\begin{rem}
 In 1970, it was shown that $$ \{n(\ell_p^2): \,\, 1 \leq p \leq \infty \}= [0,1]$$ in the real case(see  \cite{JD-CM-JP-AW}). Using the above result of Ed-dari and Khamsi \cite{ee-2} one can deduce 
 $$ \displaystyle \frac{M_p}{2} \leq n(\ell_p^{2}) \leq M_p$$ where  $M_p=\sup_{t\in [0,1]} \displaystyle \frac{|t^{p-1} -t|}{1+t^p}$. In \cite{MM-JM} the authors improve the lower bound  to $ \max \{ \displaystyle \frac{M}{2^{1/p}}, \frac{M}{ 2^{1/q}} \}$ for the numerical index of the two-dimensional real $L_p$-space. 
\end{rem}

Recently Martin, Meri and Popov  \cite{mm-jm-mp}  gave a lower estimate for the numerical index of the real $\L_p(\mu)$- spaces, proving the following:

\begin{thm}
Let $p\neq 2$  and   $\mu$ any positive measure. Then  $n(L_p(\mu)) > 0$  in the real case.
\end{thm}

The idea of the proof of the above theorem is to define a new seminorm on $ \mathcal{L}(L_{p}(\mu))$, which is in between the numerical radius and the operator norm i.e., the numerical radius and the operator norm are equivalent on  $\mathcal{L}(L_{p}(\mu))$ for every $p\neq 2$ and every positive measure $\mu$. This also answers the question posed in \cite{VK-MM-RP}. 

Given an operator $T\in \mathcal{L}(L_{p}(\mu)) $,  they define the \textit{absolute numerical radius} of $T$  to be
$$ |\nu|(T): = \sup \{ \displaystyle \int_{\Omega} | x^{ \sharp}Tx | d\mu:\,\,\, x\in L_p(\mu),\, ||x||_p =1\}=  \sup \{ \displaystyle \int_{\Omega} | x^{ p-1}| |Tx | d\mu:\,\,\, x\in L_p(\mu), \,||x||_p =1\}$$
where for any $x \in L_{p}(\mu)$, by $x^{\sharp}$ we mean  $x^{\sharp} = |x|^{p-1} \mbox{sign}(x)$ for the real case and  $x^{\sharp} = |x|^{p-1} \mbox{sign}(\bar{x})$ for the complex case. Note that with the above notation, for $x\in L_{p}(\mu)$ we can write:
$$ \nu(T)= \sup \{| \displaystyle \int_{\Omega}  x^{ \sharp}Tx  d\mu|:\,\,\, x\in L_p(\mu),\, ||x||_p =1\} $$
clearly $\nu(T) \leq |\nu|(T)\leq ||T||$ and for positive operators on $L_p(\mu)$ the numerical radius and the absolute numerical radius coincide. Using the absolute numerical radius, it is proved in \cite {mm-jm-mp} that:

$$ \nu(T) \geq \displaystyle \frac{M_p}{ 6}| \nu|(T)\,\,\mbox{ and}\,\,  |\nu|(T) \geq \displaystyle \frac{n( L_{p}^{\mathbb{C}}(\mu))}{2} || T ||$$

where  $M_p= \displaystyle \max_{t\in [0,1]} \displaystyle \frac{|t^{p-1}- t|}{1+t^p}$
 and  $n( L_{p}^{\mathbb{C}} (\mu)) $ denote the complex numerical index of $L_p (\mu)$ . Since $n(X) \geq \displaystyle \frac{1}{e}$ for any complex space they obtain 
 $$ n(L_p(\mu)) \geq \displaystyle \frac{M_p}{12e}.$$ Note that $M_p> 0$ for $p \neq 2$, this extends the result that $n(\ell_p^{m}) > 0$
 for $p\neq 2$ and $m\in \mathbb{N}$.
 
 It is an open problem whether or not  $n(L_p(\mu))= M_p$ in the real case, and the  value for $n(L_p(\mu))$ in the complex case is unknown. However using the concept of absolute numerical radius, one can define \textit{absolute numerical index} by:
 
 $$ |n|(L_p (\mu) ):= \inf \{ |v|(T): \,\, T\in \mathcal{L}(L_p(\mu)), \,\,||T||=1 \}$$ and in \cite{MM-JM-MP} it is shown that 
 \begin{thm}
  Let $1 < p < \infty$ and let $\mu$ be a positive measure such that $\dim (L_p(\mu)) \geq 2$. Then 
  
  $$  |n|(L_p (\mu) )= \displaystyle \frac {1}{ p^{1/p} q^{1/q} }.$$
  where $q = \displaystyle \frac{p}{p-1}$ is the conjugate exponent to $p$.
 \end{thm}
 Using the above theorem together with the estimate $\nu(T) \geq \displaystyle \frac{M_p}{ 6}| \nu|(T) $ given in \cite{mm-jm-mp}, one can improve the estimation of the numerical index of $L_p(\mu)$ space in the real case by: $$\nu(L_p(\mu)) \geq \displaystyle \frac{M_p  }{ 6}  .\displaystyle \frac {1}{ p^{1/p} q^{1/q} }.$$

 The numerical  radius of rank one operators and finite rank operators in $\L_p(\mu)$ are also considered in several papers \cite{MC-MM-JM} , \cite{MM-JM-MP}. The motivation to consider rank-one operators stems from several directions, one such reason being the possibility of getting more information about the geometry of the Banach space. For example a well-known theorem due to James states that a Banach space is reflexive if and only if each rank-one operator attains its norm. A version of this theorem given in \cite{MA-GR}  asserts that numerical radius characterizes reflexive Banach spaces for which rank-one operators attain their numerical radius. Additionally, rank-one operators are contained in any operator ideal, thus  rank-one numerical index becomes an upper bound for any numerical index defined as associated to any operator ideal. Now we mention briefly some  recent results about rank-one numerical indices. 
 
 The \textit{rank-one numerical index} of a Banach space $n_1(X)$ is defined as:
 $$ n_1(X)= \inf\{ \nu(T):\,\,\, T\in L(X), \,\,||T|| =1, \,\, \dim(T(X)) \leq 1\}$$
 \begin{thm} (\cite{MC-MM-JM},  Theorem 2.1)
 
 Let $X$ be a real Banach space. Then $$n_1(X) \geq \displaystyle \frac{1}{e}.$$
 \end{thm}
 The fact that $n(X) \geq \displaystyle \frac{1}{e}$ in the complex case is known as the Bonehnblust-Karlin theorem \cite{BG}, and the accomplishments in the above theorems is the extension to real spaces. Furthermore, this lower bound for the rank-one numerical index is the best possible bound, as shown in  \cite{MC-MM-JM}.  Many properties of rank-one numerical indices follow from analogous  ones for the classical numerical index. For example  for the cases $c_0$-, $\ell_1$-, or $\ell_{\infty}$-sums  we have the following:
 
 \begin{pro}
 Let   $\{X_i \,\, i\in I\}$  be a of Banach spaces. Then

$$n_1([\oplus_{i\in I} X_i]_{c_{0}}) = n_1([  \oplus_{i\in I} {X_i}]_{\ell _{1}})=n_1([  \oplus_{i\in I} {X_i}]_{\ell _{\infty}})  =  \displaystyle \inf _{i} n_1(X_i).$$
 \end{pro}
It is also known that rank-one numerical index of an absolute sum of Banach spaces is less or equal than the rank-one numerical index of each summands, as stated in the following proposition.
\begin{pro}

(\cite{mm-jm-mp-br})
Let  $\{X_i \,\, i\in I\}$  be a family of Banach spaces and $1< p< \infty. $ Then,

$$ n_1([\oplus_{i \in I} X_i]_{\ell_p} )\leq  \inf \{ n_1(X_i) \,\,\,  i\in I \} $$
 
\end{pro}
However, rank-one numerical indices of vector valued spaces exhibit a different behavior from that of  classical numerical indices. In fact we have  :

\begin{pro} (\cite{MC-MM-JM}, Proposition 2.6) Let $X$ be a Banach space and  $K$ a compact Hausdorff space. Then,
$n_1(C(K,X))=1$ when $K$ is perfect and $n_1(C(K,X))= n_1(X)$ when $K$ is not perfect.
\end{pro}

Similar results also hold true for $n_1(L_{1}(\mu, X))$ and for $n_1(L_{\infty} (\mu, X))$ .

Of course one can extend the definition of rank-one numerical index  $n_1(X)$ to the  rank-r numerical index $n_r(X)$ in a natural way by  setting $$n_r(X):=  \inf\{ \nu(T):\,\,\, T\in L(X), \,\,||T|| =1, \,\, \dim(T(X)) \leq r\}$$ or to the compact numerical index by $$n_{comp}(X) :=  \inf\{ \nu(T):\,\,\, T\in L(X), \,\,||T|| =1, \,\,  T  \,\, \mbox{compact}\}. $$Then the inequalities:

$$ n_{r}(X) \geq n_{r+1} (X)  \geq n_{comp}(X) \geq n(X)$$  follow easily. Additionally  an example of a Banach space $X$ for which $n(X) < n_{comp}(X) < n_1(X)$  also exists \cite{MC-MM-JM}.
\begin{rem}
Theorems $2.12$ and $2.13$ can be restated with the same proof for $n_r(.)$  and for $n_{comp}(.)$. For example, if the norm on $X$ satisfies LCC, or equivalently  when
 $X$ and $X_m $ and $ P_m$ be as in Definition (\ref{defLCC}), then we have:
$$
n_r(X) = \displaystyle \lim_m n_{r}(X_m).
$$

Similarly, when the norm on $X$ satisfies GCC, or equivalently, when $X$ and $X_s $ and $ P_s$ be as in Definition (\ref{defGCC}), then we have:
$$
n_r(X) = \displaystyle \lim_s n_r(X_s).
$$
where  $S$ is any directed, infinite set. 
\end{rem}
Perhaps it is  also worth mentioning that if we let $A$ be a $C^*$-algebra, then the numerical index of $A$ as an algebra is $1$ or $1/2$ depending on whether $A$ is commutative or not commutative \cite{mjc-jd-cmm}. Later T. Huruya   \cite{th} showed that the numerical index of $A$ as a normed space is also $1$ or $1/2$ depending on $A$ is commutative or not commutative.
 
Note the we do not know the value of the numerical index of  the space of $k$-times differentiable functions on $[0,1]$, namely $C^{k}[0,1]$, or Lorentz or Orlicz spaces.
\begin{rem}
There are many Banach spaces whose numerical indices are known, for example the  following classical Banach spaces  all have numerical index $1$, which is the largest possible value,(see  \cite{JD-CM-JP-AW})
 $$c_0, \,\,c, \,\, \ell_1, \,\, \ell_{\infty},\,\, C(K)\,\mbox{for every compact K}\, ,\, \mbox{and } \,\,  L_1(\mu)$$
 all function algebras like $A(\mathbb{D}), \,\, H^{\infty}$ (see \cite {dw}    ) and finite co-dimensional subspaces of $C[0,1]$ (see \cite{kb-vk-mm-dw}). Furthermore, in \cite{JD-CM-JP-AW}  two families of Banach spaces with numerical index  $1$ are identified, these are L-spaces and M-spaces.
 \end{rem}
Recall that a Banach space $X$ is said to have numerical index $1$ if and only if , the norm of $T$ , for every $ T \in B_{X}$ can redefined as:
\begin {equation} \label{ni1}
 || T || =\sup  \{| x^{*}(Tx) | :~x\in S_{X},~x^{*}\in S_{X^{*}},~x^{*}(x)=1\}.
 \end{equation} 
 The above form of the norm of $T$ given in \ref{ni1}  has consequences on the geometry or the topology of the Banach space. For example, is it possible to re-norm a given infinite dimensional Banach space to have numerical index one?  There has been a considerable amount of research  done in this direction.  Recall that a space $X$ is said to be strictly convex when $ext(B_X)= S_X$ and it is well known that if $X^*$ is strictly convex (resp. smooth) then $X$ is smooth (resp. strictly convex), but the converse is not true. Furthermore, the  norm of $X$ is said to be Fr\'{e}chet smooth when the norm of $X$ is Fr\'{e}chet differentiable at any point of $S_X$.
  There are several results which shows that Banach spaces with numerical index one  cannot enjoy good convexity or smoothness properties unless they are one-dimensional. In particular, such spaces have no  Weakly Local Uniform Rotundity (WLUR) points in their unit ball, their norm is not Fr\'{e}chet smooth, and their dual norm is neither strictly convex nor smooth. For a detailed study of these we refer the reader to \cite{vk-mm-jm-rp}.

Furthermore, it is well known that  (see \cite{bff-dj1}), $\nu(T^*)= \nu(T)$  for every $T\in B(X)$ where  $T^*$ is the adjoint operator of $T$ which implies that $n(X^*) \leq n(X)$ for every Banach space $X$, since  the map $ T \rightarrow T^*$ is a linear isometry from $B(X)$ into $B(X^*)$

$$ n(X^*) := \inf\{\nu(S): \,\, S\in B(X^*), ||S|| =1\} \leq \inf \{ \nu(T^*):\,\, T\in B(X), ||T|| =1 \}=n(X)$$
for  a reflexive space $X$,  since $n(X)=n(X^{**}) \leq n(X^*)$ clearly $n(X)=n(X^*)$. Following is an  example taken from  \cite{kb-vk-mm-dw}, which proves the fact  that the numerical index of the dual of a Banach space can be strictly smaller than the numerical index of the space. 

\begin{exm}[\cite{kb-vk-mm-dw}, Example 3.1] There exists a Banach space $X$ such that $n(X)=1$ and $n(X^*) < 1$. where  $$ X=\{(x,y,z)\in c \oplus_{\infty} c \oplus_{\infty} c : \,\, \lim x + \lim y+ \lim z =0\}.$$ 

\end{exm}

Recall that a Banach space is said to be L-embedded if $X^{**} = X \oplus X_s$ for some closed subspace $X_s$ of $X^{**}$. There are  some  recent positive results concerning numerical index of the space $X$ and its dual $X^*$. For example, it is known that $n(X)=n(X^*)$   when $X$ is L-embedded in $X^{**}$.  If $n(X)=1$ and $X$ is M-embedded in $X^{**}$ then also $n(X^*)=1$ and $n(Y)=1$ for every $ X\subset Y \subset X^{**}$ \cite{MM}. Moreover, $n(X)=n(X^*)$  holds true when $X$ is a $C^*$-algebra \cite {VK-MM-RP}.  
Another interesting result is the following: if the dual of a Banach space $X$ has the Radon Nikodym Property (RNP) and $n(X)=1$, then $n(X^*)=1$  as well  (see \cite{kb-vk-mm-dw} for details). It is still an open question whether $n(X)=n(X^*)$ holds true for every Banach space with the RNP. 

\begin{rem}
The\textit { polynomial numerical index of order $k$ } of a Banach space was first introduced in \cite{ysc-dg-sgk-mm}. Let $E$ and $F$ be real or complex Banach spaces and $\mathcal{L}(^kE:F)$ denote the Banach space of continuous, $k$-linear mappings $E^k := E\times \dots \times E$ into $F$. $P:E \to F$ is called continuous $k$-homogeneous polynomial if there is $A \in \mathcal{L}(^kE:F)$ such that $P(x)= A(x,\dots, x)$ for all $x\in E$.  Let $\mathcal{P} (^kE:F)$ denote the Banach space of continuous $k$-homogeneous polynomials of $E$ into $F$, endowed with the polynomial norm $||P|| = \displaystyle \sup_{x\in B_E} ||P(x)||$. 

For each  $P \in  \mathcal{P} (^kE:F)$, the numerical range of $P$ is the subset $W (P)$ of the scalar field defined by
 
$$W(P)= \{ x^{*}(Px)  :~x\in S_{X},~x^{*}\in S_{X^{*}},~x^{*}(x)=1\}$$

 and the numerical radius of P is given by
$$\nu(P ) = \sup\{|\lambda| : \lambda  \in  W (P )\} .$$
The polynomial numerical index of order $k$ of the space $E$, $n^{(k)} (E)$ is defined as: 
$$ n^{(k)}(E) = \mbox{inf} \{\nu(P): \,\,\, P\in S_{\mathcal{P} (^kE:E)} \,\,\, \}$$ or equivalently it 
is the greatest constant $c$ such that $$ c||P|| \leq \nu(P)$$ for every $P\in \mathcal{P} (^kE:E) $. Note that polynomial numerical index of order $1$ is the ``classical" numerical index and just like in the case of the classical numerical index, we have: $0 \leq n^{(k)} (E) \leq 1$ and $n^{(k)}(E) >0$ if and only if $|| . ||$ and $\nu( . )$ are equivalent norms on $\mathcal{P} (^kE:E)$. Clearly if $E_1$ and $E_2$ are isometrically isomorphic Banach spaces , then $n^{(k)} (E_1) = n^{(k)}(E_2)$ hold. For more on the properties of the polynomial numerical index of order $k$, we refer the reader to \cite{dg-bcg-mm-mm-jm}, \cite{hjl-mm-jm}. \cite{ysc-dg-sgk-mm} and \cite{sgk}.
\end{rem}


\noindent
\mbox{~~~~~~~}Asuman G\"{u}ven AKSOY\\
\mbox{~~~~~~~}Claremont McKenna College\\
\mbox{~~~~~~~}Department of Mathematics\\
\mbox{~~~~~~~}Claremont, CA  91711, USA \\
\mbox{~~~~~~~}E-mail: aaksoy@cmc.edu \\ \\

\noindent
\mbox{~~~~~~~}Grzegorz LEWICKI\\
\mbox{~~~~~~~}Jagiellonian University\\
\mbox{~~~~~~~}Department of Mathematics\\
\mbox{~~~~~~~}\L ojasiewicza 6, 30-348, Poland\\
\mbox{~~~~~~~}E-mail: Grzegorz.Lewicki@im.uj.edu.pl\\\\

\end{document}